\documentclass[12pt]{article}
\usepackage[cmtip,arrow]{xy}
\usepackage{pb-diagram, pb-xy}
\usepackage{amssymb,epsfig,amsfonts,amsmath}

\newtheorem{prop}{Proposition}

\newtheorem{Th}{Theorem}
\newtheorem{lemma}{Lemma}

\newfont{\ssdbl}{msbm8}
\newfont{\sdbl}{msbm9}
\newfont{\dbl}{msbm10 at 12pt}

\newcommand{\oo}{{\cal O}}
\newcommand{\ff}{{\cal F}}

\newcommand{\ad}{{\cal A}}

\newcommand{\res}{\mathop {\rm res}}

\newcommand{\tr}{\mathop {\rm Tr}}

\newcommand{\da}{\mathbb{A}}
\newcommand{\dz}{\mathbb{Z}}
\newcommand{\dc}{\mathbb{C}}
\newcommand{\lto}{\longrightarrow}

\newcommand{\F}{{\bf F}}

\newcommand{\D}{{\cal D}}

\def\C{{\mathbb C}}

\newcommand{\Div}{{\rm Div}}

\begin{document}

\author{D.V. Osipov, A.N. Parshin
\footnote{Both authors are supported by RFBR (grants
no.~11-01-00145-a and no.~11-01-12098) and by a program of President
of RF for supporting of Leading Scientific Schools (grant
no.~NSh-4713.2010.1).}  }

\title{Harmonic analysis and the Riemann-Roch theorem.}
\date{}

\maketitle


\noindent {\bf 1.} \ Let $D$ be a smooth projective curve over a
finite field $k$. It is known (see, e.g.,~\cite[\S3]{P1}) that the
Poisson summation formula applied to the discrete subgroup  $k(D)$
of the adelic space $\da_D$ implies the Riemann-Roch theorem on the
curve $D$. This result is an important step in the application of
harmonic analysis to the arithmetic of algebraic curves. In this
note we show, how to solve the analogous problem for the case of
dimension two. Namely, we will show how the Riemann-Roch theorem for
invertible sheaves  on a projective smooth algebraic surface $X$
over $k$ (in a variant without the Noether formula, see,
e.g.,~\cite{S}) is obtained from the two-dimensional Poisson
formulas (see~\cite[\S5.9]{OsipPar1} and \cite[\S13]{OsipPar2}).

First, we need some general proposition. Let $E = (I,F, V)$ be a $C_2$-space over the field $k$ (see~\cite{Osip}).
 Recall that for any $i,j \in I$
we have constructed in~\cite[\S5.2]{OsipPar1} a one-dimensional
$\C$-vector space of virtual measures $\mu (F(i) \mid F(j)) = \mu
(F(i) / F(l))^* \otimes_{\C} \mu (F(l) / F(j))$, where $l \in I$
such that $l \le i$, $l \le j$, and  $\mu(H)$ is the space of
$\C$-valued
 Haar measures on a $C_1$-space $H$.
The space $\mu (F(i) \mid F(j))$ does not depend on the choice of $l \in I$ up to a canonical isomorphism.

Let $0 \to A \to E \to B \to 0$ be an admissible triple of
$C_2$-spaces over $k$. Let $A = (J,G, W)$ as a $C_2$-space, and $W
=F(j)$ for some $j \in I$. Then $A$ is a $cC_2$-space and $B$ is a
$dC_2$-space (see~\cite[\S5.1]{OsipPar1}).  Let $o \in I$ and $\mu
\in \mu(W/ F(o) \cap W)$, $\nu \in \nu(F(o)/ F(o) \cap W)^*$. Then
in~\cite[form.~(164)]{OsipPar1} we have constructed the
characteristic element $\delta_{A, \mu \otimes \nu} \in
\D'_{F(o)}(E)$. We note that $\mu \otimes \nu  \in \mu (F(o) \mid
W)$. Therefore we can replace $\mu \otimes \nu$ by $\eta \in \mu
(F(o) \mid W)$ and write $\delta_{A, \eta} \in \D'_{F(o)}(E)$
instead of the previous notation.

Let $0 \to L \to E \to M \to 0$  be an admissible triple of
$C_2$-spaces over $k$ such that $L$ is a $cfC_2$-space and $M$ is a
$dfC_2$-space (see~\cite[\S5.1]{OsipPar1}).
In~\cite[form.~(169)]{OsipPar1} we have constructed the
characteristic element $\delta_{L} \in \D_{F(o)}(E)$. We note that
for any $i,j \in I$ the space $L$ defines a non-zero element
$\mu_{L, F(i), F(j)} \in \mu (F(i) \mid F(j))$ in the following way.
Let $L =(K, T, U)$ as a $C_2$-space. Choose some $l \in I$ such that
$l \le i$, $l \le j$. Then $\mu_{L, F(i), F(j)} = \mu_{L, F(l),
F(i)}^{-1} \otimes \mu_{L, F(l), F(j)}$, where for any $m \le n \in
I$ we define $\mu_{L,m,n} \in \mu(F(n)/F(m)) $ as $\mu_{L,m,n}( U
\cap F(n)/ U \cap F(m))=1$. The element $\mu_{L, F(i), F(j)}$ does
not depend on the choice of $l \in I$.

There is a natural pairing $< \cdot , \cdot > : \D_{F(o)}(E) \times
D'_{F(o)}(E) \to \C$. From the above definitions it is easy to prove
the following proposition.
\begin{prop} \label{prc}
$$
<\delta_L, \delta_{A, \eta} > = \frac{\eta}{\mu_{L, F(o), W}}
\mbox{.}
$$
\end{prop}


\vspace{0.5cm} \vspace{0.5cm} \noindent {\bf 2.} \
 Let $X$ be a
smooth projective algebraic  surface over a finite field $k$. Let
$\mid k \mid = q$. For any quasicoherent sheaf $\ff$ on $X$ there is
an adelic complex $\ad_X (\ff)$ such that $H^* (\ad_X(\ff)) = H^*
(X, \ff)$. Let $C \in \Div(X)$. For the sheaf $\oo_X(C)$ on $X$  we
will write this complex  in the following way:
$$
\da_{0,C} \oplus \da_{ 1,C} \oplus \da_{2,C} \lto \da_{01,C} \oplus
\da_{ 02, C} \oplus \da_{ 12, C}
 \lto
 \da_{ 012, C}  \mbox{,}
 $$
 where $\da_{ *, \, C} = \da_{X, *}(\oo_X(C))$ (see the corresponding notations and definitions
 in~\cite[\S 14.1]{OsipPar2}),
and we have omitted  indication on $X$ in the notations of subgroups
of the adelic complex, because we will work only with one algebraic
surface $X$ during this note.
  We note that that all the groups $\da_{ *, C}$ are subgroups of the group $\da_{012, C}$.
  Besides, the following groups does not depend on $C \in \Div(X)$:
  $$\da_{ 0, C} = \da_{ 0} \mbox{,} \quad  \da_{01,C}= \da_{ 01} \mbox{,}
  \quad \da_{02,C}= \da_{ 02} \mbox{,} \quad  \da_{012,C}= \da_{ 012} = \da \mbox{.}$$
Moreover, $\da \subset \prod_{x \in D} K_{x,D}$, where $x \in D$
runs over all pairs with irreducible curve
 $D$ on $X$ and $x$ is a point on $D$.
The ring $K_{x,D}$ is a finite product of two-dimensional local fields with the last residue field $k(x)$.

We fix a non-zero rational differential form $\omega \in
\Omega^2_{k(X)/k}$. Let $(\omega) \in \Div(X)$ be the corresponding
divisor. The following pairing (which depends on $\omega$) is
well-defined, symmetric   and non-degenerate:
\begin{equation}   \label{pa}
\da \times \da \lto k \quad : \quad  \{ f_{x,D}\} \times \{ g_{x,D}
\}
 \mapsto \sum_{x \in D}  \tr\nolimits_{k(x)/k} \circ \res\nolimits_{x,D} (f_{x,D}  \, g_{x,D}  \, \omega)  \mbox{,}
\end{equation}
where $\res_{x,D}$ is the two-dimensional residue. For any
$k$-subspace $V \subset \da$ we will denote by $V^{\bot} $ the
annihilator   of $V$ in $\da$ with respect to the
pairing~\eqref{pa}. Using the reciprocity laws for the residues of
differential forms on $X$ (the reciprocity laws "around a point" and
the reciprocity laws "along a curve") one can prove the following
proposition.
\begin{prop} \label{prd}
We have the following properties.
\begin{gather*}
\da_{ 0}^{\bot} = \da_{01} + \da_{ 02} \mbox{,} \quad
 \da_{1,C}^{\bot} = \da_{ 01} + \da_{12, (\omega) -C } \mbox{,}
 \quad
 \da_{2,C}^{\bot} = \da_{ 02} + \da_{12, (\omega) -C }
 \\
 \da_{ 01}^{\bot} = \da_{01} \mbox{,}  \quad
 \da_{ 02}^{\bot} = \da_{ 02} \mbox{,}
 \quad
 \da_{ 12,C}^{\bot} = \da_{12, (\omega) -C } \mbox{.}
  \end{gather*}
\end{prop}

We note that $\da = \mathop{\lim\limits_{\longrightarrow}}\limits_{C
\in \Div(X)} \da_{12,C} $, and $\da_{12,C} =
\mathop{\lim\limits_{\longleftarrow}}\limits_{C' \le C } \da_{12,C}/
\da_{12,C'}$. For any $C' \le C$ the $k$-space
$\da_{12,C}/\da_{12,C'}$ has the natural structure of  a complete
$C_1$-space over the field $k$. Hence we obtain that the $k$-space
$\da$ has the following structure of a complete $C_2$-space over
$k$: $(\Div(X), F, \da)$, where $F(C)= \da_{12,C}$ for $C \in
\Div(X)$. For simplicity  we will use the same notation $\da$ for
this $C_2$-space, i.e. we will omit the partially ordered set
$\Div(X)$ and the function $F$. The subspaces $\da_{*, C}$ of $\da$
(and the factor-spaces by these subspaces) have induced structures
of $C_2$-spaces, which we will also denote by the same notations
$\da_{*,C}$ (by notations for factor-spaces).

From proposition~\ref{prd} it follows that the $C_2$-dual space
(see~\cite[\S 5.1]{OsipPar1}) $\check{\da}$ coincides with the
$C_2$-space $\da$ itself:
$$
\check{\da} = \mathop{\lim\limits_{\longleftarrow}}\limits_{C \in
\Div(X)} \mathop{\lim\limits_{\longrightarrow}}\limits_{C' \le C }
\da_{12,C'}^{\bot}/\da_{12,C}^{\bot} =
\mathop{\lim\limits_{\longrightarrow}}\limits_{C' \in \Div(X)  }
\mathop{\lim\limits_{\longleftarrow}}\limits_{C \ge C' }
\da_{12,(\omega) - C'}/\da_{12,(\omega) - C} = \da \mbox{.}
$$


\vspace{0.5cm} \noindent {\bf 3.} \ For any $E \in \Div(X)$ we
denote  $h^i(E)= \dim_k H^i(X, \oo_X(E))$, where $0 \le i \le 2$. We
fix any $H, C \in \Div(X)$. We consider the following admissible
triple of complete $C_2$-spaces over $k$:
\begin{equation}  \label{s1}
0 \lto \da_{0} \lto \da_{01} \lto \da_{ 01}/\da_{0} \lto 0 \mbox{.}
\end{equation}
The space $\da_{0}$ is a $cfC_2$-space, and the space $\da_{
01}/\da_{0} $ is a $dfC_2$-space. Therefore there is the
characteristic element $\delta_{\da_{0}} \in \D_{\da_{1,H}}(\da_{
01})$.

Now we consider the following admissible triple of complete
$C_2$-spaces over $k$:
\begin{equation}  \label{s2}
0 \lto \da_{1,C} \lto \da_{ 01} \lto \da_{01}/\da_{1,C} \lto 0
\mbox{.}
\end{equation}
 We note that the space
$\da_{01}$ is a $dfC_2$-space. Therefore for any $H', C' \in
\Div(X)$ there is a natural element  $\delta_{H', C'} \in \mu(
\da_{1,H'} \mid \da_{1,C'})$ which is uniquely defined by the
following two conditions: 1) $\delta_{H', M'} \otimes \delta_{M',
C'} = \delta_{H', C'}$ for any $H', M', C' \in \Div(X)$, and 2) if
$H' \le C'$ then $\delta_{H',C'} \in \mu(\da_{1,C'}/ \da_{1,H'}) $
is defined as $\delta_{H',C'}( (0))=1$, where $(0)$ is the zero
subspace in the discrete $C_1$-space $ \da_{1,C'}/ \da_{1,H'}$.
Besides, the space $\da_{1,C}$ is a $cC_2$-space, and the space
$\da_{01}/\da_{1,C}$ is a $dC_2$-space. Hence there is the
characteristic element $\delta_{\da_{1,C}, \, \delta_{H,C}}  \in
\D'_{\da_{1,H}}(\da_{ 01})$.
\begin{lemma} \label{lem1}  We have the following equality:
$$
<\delta_{\da_{0}} \, , \, \delta_{\da_{1,C}, \, \delta_{H,C}} > =
q^{h^0(C)- h^0(H)}  \mbox{.}
$$
\end{lemma}
\begin{proof}
We will use proposition~\ref{prc}. From this proposition it follows
that it is enough to consider $H \le C$. In this case, by this
proposition again, we have $<\delta_{\da_{0}}, \delta_{\da_{1,C}, \,
\delta_{H,C}} > = q^{\dim_k V}$, where the $k$-vector space $V =
(\da_{0} \cap \da_{1,C})/ (\da_{0} \cap \da_{1,H}) $. Now we use
$\da_{0} \cap \da_{1,E} = H^0(X, \oo_X(E)) $ for any $E \in
\Div(X)$. The lemma is proved.
\end{proof}

Now we fix any $P, Q \in \Div(X)$.
We consider the following admissible
triple of complete $C_2$-spaces over $k$:
\begin{equation} \label{d1}
0 \lto \da_{02}/ \da_{0} \lto \da/ \da_{ 01} \lto \da/(\da_{02} +
\da_{01}) \lto 0 \mbox{,}
\end{equation}
where we use that $\da_{0} = \da_{01} \cap \da_{02}$. The space
$\da_{02}/\da_{0}$ is a $cfC_2$-space, and the space $\da/(\da_{ 01}
+ \da_{02}) $ is a $dfC_2$-space. Therefore there is the
characteristic element $\delta_{\da_{02}/\da_{0}} \in
\D_{\da_{12,P}/ \da_{1,P}}(\da /\da_{ 01})$.

Now we consider the following admissible triple of complete
$C_2$-spaces over $k$:
\begin{equation} \label{d2}
0 \lto \da_{12,Q}/ \da_{1,Q}  \lto \da/ \da_{ 01} \lto
\da/(\da_{12,Q} + \da_{01}) \lto 0   \mbox{,}
\end{equation}
where we use that $\da_{1,Q} = \da_{01} \cap \da_{12,Q}$.
 We note that the space
$\da/\da_{ 01}$ is a $cfC_2$-space. Therefore for any $P', Q' \in
\Div(X)$ there is the following natural element  $1_{P', Q'} \in
\mu( \da_{12,P'}/\da_{1,P'} \, \mid \, \da_{12,Q'}/\da_{1,Q'})$
which is uniquely defined by the following two conditions: 1)
$1_{P', R'} \otimes 1_{R', Q'} = 1_{P', Q'}$ for any $P', R', Q' \in
\Div(X)$, and 2) if $P' \le Q'$ then $1_{P',Q'} \in
\mu((\da_{12,Q'}/\da_{1,Q'})/ (\da_{12,P'}/\da_{1,P'})) $ is defined
as $1_{P',Q'}((\da_{12,Q'}/\da_{1,Q'})/ (\da_{12,P'}/\da_{1,P'})
)=1$, since $(\da_{12,Q'}/\da_{1,Q'})/ (\da_{12,P'}/\da_{1,P'})$ is
a compact $C_1$-space. Besides, the space $\da_{12,Q}/ \da_{1,Q} $
is a $cC_2$-space, and the space $\da/(\da_{12,Q} + \da_{01})$ is a
$dC_2$-space. Hence there is the characteristic element
$\delta_{\da_{12,Q} / \da_{1,Q} , \, 1_{P,Q}} \in
\D'_{\da_{12,P}/\da_{1,P}}(\da/\da_{ 01})$.
\begin{lemma} \label{lem2} We have the following equality:
$$
<  \delta_{\da_{02}/\da_{0}}     \,  , \,  \delta_{\da_{12,Q} /
\da_{1,Q} , \, 1_{P,Q}}
 > = q^{h^2(Q)- h^2(P)}
\mbox{.}
$$
\end{lemma}
\begin{proof}
We will use proposition~\ref{prc}. By this proposition, it is enough
to consider $P \ge Q$. In this case, by this proposition again, we
have   $ <  \delta_{\da_{02}/\da_{0}} \, , \, \delta_{\da_{12,Q} /
\da_{1,Q} , \, 1_{P,Q}}
 >
= q^{\dim_k W}$, where the $k$-vector space  $W = (\da_{01} +
\da_{02} + \da_{12,P})/(\da_{01} + \da_{02} + \da_{12,Q}) $. Now we
use that from the adelic complex $\ad_X(\oo_X(E))$ we have
$\da/(\da_{01} + \da_{02} + \da_{12,E}) = H^2(X, \oo_X(E)) $ for any
$E \in \Div(X)$. The lemma is proved.
\end{proof}

Now we suppose that $Q = (\omega) - C$ and $P = (\omega) -H $. From
proposition~\ref{prd} it follows that triple~\eqref{d1} is a
$C_2$-dual sequence to triple~\eqref{s1}, and triple~\eqref{d2} is a
$C_2$-dual sequence to triple~\eqref{s2}.  We have also the
two-dimensional Fourier transforms $\F : \D_{\da_{1,H}}(\da_{ 01})
\to \D_{\da_{12,P}/ \da_{1,P}}(\da /\da_{ 01})$  and  $\F :
\D'_{\da_{1,H}}(\da_{ 01}) \to \D'_{\da_{12,P}/ \da_{1,P}}(\da
/\da_{ 01})$ (see~\cite[\S 5.4.2]{OsipPar1} and~\cite[\S
8.2]{OsipPar2}), which we denote by the same letter, although they
act from various spaces.  Now by the two-dimensional Poisson formula
II (see~\cite[th.~3]{OsipPar1}) we have $\F (\delta_{\da_{0}})=
\delta_{\da_{02}/\da_{0}}$. By the two-dimensional Poisson formula I
(see~\cite[th.~2]{OsipPar1}) we have $\F (\delta_{\da_{1,C}, \,
\delta_{H,C}}) = \delta_{\da_{12,Q} / \da_{1,Q} , \, 1_{P,Q}}$. (We
used  that according to~\cite[form.~(103)]{OsipPar1} we have $\mu(
\da_{1,H} \mid \da_{1,C}) = \mu( \da_{12,P}/\da_{1,P} \, \mid \,
\da_{12,Q}/\da_{1,Q})$, and $ \delta_{H,C} \mapsto 1_{P,Q}$ under
this isomorphism.) Now since $\F \circ \F (g)= g$ for $g =
\delta_{\da_{0}}$ or $g = \delta_{\da_{1,C}, \, \delta_{H,C}}$, and
the maps $\F$ are conjugate with respect to each other
(see~\cite[prop.~24]{OsipPar1}), we have that $<\delta_{\da_{0}} ,
\delta_{\da_{1,C}, \, \delta_{H,C}} > =< \F (\delta_{\da_{0}})    ,
\F (\delta_{\da_{1,C}, \, \delta_{H,C}})
>$. Hence and from lemmas~\ref{lem1}-\ref{lem2} we obtain for any $H, C \in
\Div(X)$ the following equality:
\begin{equation} \label{eq1}
h^0(C) - h^0(H) = h^2((\omega) - C) - h^2((\omega)- H) \mbox{.}
\end{equation}


\vspace{0.5cm} \noindent {\bf 4.} \ For any $E \in \Div(X)$ we
denote the Euler characteristic $\chi(E)= h^0(E) - h^1(E) + h^2(E)$.
We fix any $R, S \in \Div(X)$. We consider the following admissible
triple of complete $C_2$-spaces over $k$:
\begin{equation}  \label{ss1}
0 \lto \da_{02} \lto \da \lto \da/\da_{02} \lto 0 \mbox{.}
\end{equation}
The space $\da_{02}$ is a $cfC_2$-space, and the space $\da/\da_{02}
$ is a $dfC_2$-space. Therefore there is the characteristic element
$\delta_{\da_{02}} \in \D_{\da_{12,R}}(\da)$.

Now we consider the following admissible triple of complete
$C_2$-spaces over $k$:
\begin{equation}  \label{ss2}
0 \lto \da_{12,S} \lto \da  \lto \da/\da_{12,S} \lto 0 \mbox{.}
\end{equation}
The subspace $\da_{01}$ uniquely defines an element $\nu_{R',S'} \in
\mu(\da_{12,R'} \mid \da_{12,S'} )$  for any $R', S' \in \Div(X)$ in
the following way. If $R' \le S'$, then we consider the following
admissible triple of $C_1$-spaces:
$$
  0 \lto \da_{1,S'}/\da_{1,R'} \lto \da_{12,S'} / \da_{12,R'}  \lto
  \da_{12,S'}/ (\da_{1,S'} + \da_{12,R'})   \lto 0 \mbox{,}
$$
where $\da_{1,S'}/\da_{1,R'}$ is a discrete $C_1$-space, and
$\da_{12,S'}/ (\da_{1,S'} + \da_{12,R'})$ is a compact $C_1$-space.
Now $\nu_{R',S'} \in \mu(\da_{12,S'}/ \da_{12,R' })$ is equal to
$\delta_0 \otimes 1$, where $\delta_0 ((0))=1$, $\delta_0 \in
\mu(\da_{1,S'}/\da_{1,R'}) $, and $1 ( \da_{12,S'}/( \da_{1,S'} +
\da_{12,R'}) ) =1$, $1 \in \mu(\da_{12,S'}/ (\da_{1,S'} +
\da_{12,R'}))$. For arbitrary $R',S'$ the element $\nu_{R',S'}$ is
defined by the following rule: $\nu_{R',S'} = \nu_{R',T'} \otimes
\nu_{T',S'}$, where $T' \in \Div(X)$ is any. The space $\da_{12,S}$
is a $cC_2$-space, and the space $\da / \da_{12,S}$ is a
$dC_2$-space. Hence there is the characteristic element
$\delta_{\da_{12,S} , \, \nu_{R,S}} \in \D'_{\da_{12,R}}(\da)$.
\begin{lemma} \label{lem3}  We have the following equality:
$$
<\delta_{\da_{02}} \, , \, \delta_{\da_{12,S}, \, \nu_{R,S}} > =
q^{\chi(S)- \chi(R)}  \mbox{.}
$$
\end{lemma}
\begin{proof}
We will use proposition~\ref{prc}. From this proposition it follows
that it is enough to consider $R \le S$. In this case, by this
proposition again, we have $<\delta_{\da_{02}}, \delta_{\da_{12,S},
\, \nu_{R,S}} > = q^a$, where $a$ is equal to the Euler
characteristic of the following complex, which has the
finite-dimensional over $k$ cohomology groups:
\begin{equation} \label{com}
\da_{1,S}/\da_{1,R}  \: \oplus \:  \da_{2,S}/\da_{2,R}  \, \lto \,
\da_{12,S} / \da_{12,R} \mbox{.}
\end{equation}
Complex~\eqref{com} is the factor-complex of the adelic complex
$\ad_X(\oo_X(S))$ by the adelic complex $\ad_X(\oo_X(R))$. Therefore
the Euler characteristic of complex~\eqref{com} is the difference of
the Euler characteristics of corresponding adelic complexes. The
lemma is proved.
\end{proof}

From proposition~\ref{prd} it follows that triple~\eqref{ss1} itself
is a $C_2$-dual sequence to triple~\eqref{ss1}, and
triple~\eqref{ss2} is a $C_2$-dual sequence to triple~\eqref{ss2}
when $S \mapsto (\omega) -S$. We have also the
two-dimensional Fourier transforms $\F : \D_{\da_{12,R}}(\da)
\to \D_{\da_{12,(\omega) -R}}(\da)$  and  $\F :
\D'_{\da_{12,R}}(\da)
\to \D'_{\da_{12,(\omega) -R}}(\da)
$.
By the two-dimensional Poisson formulas (see~\cite[th.~2-th.~3]{OsipPar1}) we have $\F (\delta_{\da_{02}}) = \delta_{\da_{02}}$ and
$\F (\delta_{\da_{12,S}, \, \nu_{R,S}})= \delta_{\da_{12, (\omega)-S} , \: \nu_{(\omega)-R,  \, (\omega)-S}} $.
(We used that from proposition~\ref{prd} it follows that $\nu_{R,S} \mapsto \nu_{(\omega)-R,  \, (\omega)-S} $
under the natural isomorphism $\mu (\da_{12,R}  \mid \da_{12,S})  = \mu (\da_{12, (\omega) -R} \mid \da_{12, (\omega) -S}) $.)
From~\cite[prop.~24]{OsipPar1} we have $<\delta_{\da_{02}} \, , \, \delta_{\da_{12,S}, \, \nu_{R,S}} > =
<\F(\delta_{\da_{02}}) \, , \, \F (\delta_{\da_{12,S}, \, \nu_{R,S}}) >$.
Hence and from lemma~\ref{lem3} we have that $\chi(S) - \chi(R)= \chi((\omega)-S) - \chi((\omega)-R)$. If we put $R = (\omega)-S$, then for any $S \in \Div(X)$ we obtain from the previous formula the following equality:
\begin{equation} \label{eq2}
\chi(S) = \chi((\omega)-S) \mbox{.}
\end{equation}

\vspace{0.5cm} \noindent {\bf 5.} \  In section {\bf 1} we introduced the element  $\mu_{L, F(i), F(j)} \in \mu(F(i),F(j)))$ for the admissible monomorphism of $C_2$-spaces $L \to E$. When $L = \da_{02}$, $E = \da$,
$F(i)= \da_{12, R}$, $F(j)= \da_{12, S}$ for $R, S \in \Div(X)$ we will denote this element by $\mu_{R,S}$. From the proof of lemma~\ref{lem3} it follows that
\begin{equation} \label{eq3}
q^{\chi(S) - \chi(R)} = \frac{\nu_{R,S}}{  \mu_{R,S}}   \mbox{.}
\end{equation}

For any $g \in \da^*$ and any $R, S \in \Div(X)$ we have a natural
action: $g^*:   \mu(\da_{12,R} \mid \da_{12,S}) \to \mu(g\da_{12,R}
\mid g\da_{12,S})$. Hence we obtain a central extension (see
also~\cite[\S5.5.3]{OsipPar1}):
$$
1 \lto \dc^* \lto \widehat{\da^*} \stackrel{\pi}{\lto} \da^* {\lto}
1 \mbox{,}
$$
where $\widehat{\da^*} = \{(g, \phi) \, : \, g \in \da^*, \, \phi \in
\mu(\da_{12, 0} \mid g \da_{12, 0}), \, \phi \ne 0 \}$, and $(g_1,
\phi_1)(g_2, \phi_2) = (g_1g_2, \phi_1 \otimes g_1^*(\phi_2))$.
(Here $\da_{12,0}$ is the  group connected with the zero divisor on
$X$.) For any $g_1, g_2 \in \da^*$ we denote $\langle g_1, g_2
\rangle = [\widehat{g_1}, \widehat{g_2}] \in \dc^*$, where
$\widehat{g_i} \in \widehat{\da^*}$ are any such that
$\pi(\widehat{g_i})=g_i$. The element $\langle g_1, g_2  \rangle$
does not depend on the choice of
 appropriate elements $\widehat{g_i}$. From~\cite{O} it follows
the following equality:
\begin{equation} \label{ext} \langle g_1, g_2  \rangle =
\prod_{x \in D} \ q^{-[k(x)\, : \,k] \, ({g_1}_{x,D} , \,
{g_2}_{x,D})_{x,D}} \mbox{,}
\end{equation}
where $(\cdot, \cdot)_{x,D}$ is the composition of the maps:
$K_{x,D}^* \times K_{x,D}^* \to K_2(K_{x,D})
\stackrel{\partial_2}{\to} \overline{K}_{x,D}^{\, *}
\stackrel{\partial_1}{\to} \dz$.

For any $E \in \Div(X)$ we choose an element $j_{1,E} \in
\da_{01}^*$ such that $\da_{1,E}= j_{1,E} \da_{1,0}$, and an element
$j_{2,E} \in \da_{02}^*$ such that $\da_{2,E}= j_{2,E} \da_{2,0}$,
where we take the product inside the ring $\da$. Now
from~\cite[\S2.2]{P} and from~\eqref{ext} it follows the following
formula for any $C, H \in \Div(X)$ ($(C,H)$ means the intersection
index of divisors $C$ and $H$ on $X$):
\begin{equation} \label{eq4}
\langle j_{2,C}, j_{1,H}  \rangle = q^{ -(C , \, H)}  \mbox{.}
\end{equation}
Since we can take $j_{1, E_1+E_2} = j_{1,E_1} j_{2,E_2}$ and $j_{2,
E_1+E_2} = j_{2,E_1} j_{2,E_2}$, we obtain $j_{1,E_1} \da_{1, E_2}=
\da_{1, E_1 + E_2}$ and $j_{2,E_1} \da_{2, E_2}= \da_{2, E_1 + E_2}$
for any $E_1, E_2 \in \Div(X)$. Hence we have $j_{1, E}^{\, *}
(\nu_{R,S}) = \nu_{R+E, \, S+E} $ and $j_{2, E}^{\, *} (\mu_{R,S}) =
\mu_{R+E, \, S+E} $
 for any $R,S, E \in \Div(X)$. For any $C \in \Div(X)$ we choose
 $\widehat{j_{2,C}}= (j_{2,C}, \nu_{0,C}) \in \widehat{\da^*}$
 and $\widehat{j_{1,(\omega) -C}}= (j_{1, (\omega)- C}, \mu_{0,(\omega)- C}) \in
 \widehat{\da^*}$. We have
 \begin{multline} \label{eq5}
\langle j_{2,C}, j_{1,(\omega)-C} \rangle = \frac{\widehat{j_{2,C}}
\; \widehat{j_{1,(\omega) -C}}}{\widehat{j_{1,(\omega) -C}} \;
\widehat{j_{2,C}}}= \frac{ \nu_{0,C} \otimes j_{2,C}^*(\mu_{0,
(\omega)-C}) }{ \mu_{0,(\omega)- C} \otimes j_{1,(\omega) -C}^*
(\nu_{0,C} )
 }=
\frac{ \nu_{0,C} \otimes \mu_{C, (\omega)} }{ \mu_{0,(\omega)- C}
\otimes \nu_{(\omega)-C,  (\omega)}
 }= \\ =
\frac{ \nu_{0,C} \otimes \mu_{C, (\omega)-C} \otimes \mu_{(\omega )-
C, (\omega)} }{\mu_{0,C} \otimes \mu_{C,(\omega)- C} \otimes
\nu_{(\omega)-C, (\omega)}
 }=
 \frac{\nu_{0,C}}{\mu_{0,C}} \; \frac{\mu_{(\omega )-
C, (\omega)}}{\nu_{(\omega)-C, (\omega)}}  \mbox{.}
\end{multline}
From~\eqref{eq3} and~\eqref{eq2} we obtain
$\frac{\nu_{0,C}}{\mu_{0,C}}= \frac{\mu_{(\omega )- C,
(\omega)}}{\nu_{(\omega)-C, (\omega)}} = q^{\chi(C) - \chi(0)}$.
Therefore from~\eqref{eq5} and~\eqref{eq4} we have
$2(\chi(C)-\chi(0))= -(C, (\omega)-C)$ for any $C \in \Div(X)$. From
the last equality and formula~\eqref{eq1} we obtain the Riemann-Roch
theorem in the following form.
\begin{Th}
For any $C \in \Div(X)$ and $\omega \in \Omega^2_{k(X)}$, $\omega
\ne 0$ we have the following equality
$$
h^0(C) - h^1(C) + h^0((\omega)-C)= h^0(0)- h^1(0)+h^0((\omega)) -
\frac{1}{2} \, (C, (\omega)-C)  \mbox{.}
$$
\end{Th}

\vspace{0.3cm}

\noindent D. V. Osipov \\  Steklov Mathematical Institute RAS \\
{\em
E-mail:}  $ \rm {d}_{-} osipov@mi.ras.ru$ \\

\noindent A. N. Parshin \\ Steklov Mathematical Institute RAS \\
{\em E-mail:}  $\rm parshin@mi.ras.ru$

\end{document}